\documentclass[a4paper, oneside, draft]{amsart}
\usepackage{amssymb}
\usepackage{amscd}
\usepackage{amsmath}
\usepackage[all]{xy}
\usepackage{mathrsfs}
 \usepackage{color}
 \usepackage[new]{old-arrows}
 
\theoremstyle{plain}
\newtheorem{theorem}{Theorem}[section]
\newtheorem{proposition}[theorem]{Proposition}
\newtheorem{lemma}[theorem]{Lemma}
\newtheorem{corollary}[theorem]{Corollary}
\newtheorem{conjecture}[theorem]{Conjecture}

\theoremstyle{remark}
\newtheorem{remark}[theorem]{Remark}

\theoremstyle{definition}
\newtheorem{definition}[theorem]{Definition}

  1

\DeclareMathOperator{\Gal}{Gal}

\DeclareMathOperator{\Hom}{Hom}
\DeclareMathOperator{\End}{End}

\DeclareMathOperator{\im}{im}

\newcommand{\bG}{\mathbb{G}}
\newcommand{\bQ}{\mathbb{Q}}

\newcommand{\bT}{\mathbb{T}}
\newcommand{\bV}{\mathbb{V}}
\newcommand{\bZ}{\mathbb{Z}}

\newcommand{\cE}{\mathcal{E}}
\newcommand{\cF}{\mathcal{F}}
\newcommand{\cG}{\mathcal{G}}

\newcommand{\cI}{\mathcal{I}}

\newcommand{\cK}{\mathcal{K}}

\newcommand{\cO}{\mathcal{O}}

\newcommand{\cR}{\mathcal{R}}

\newcommand{\fa}{\mathfrak{a}}

\newcommand{\fq}{\mathfrak{q}}
\newcommand{\fp}{\mathfrak{p}}
\newcommand{\fr}{\mathfrak{r}}

\newcommand{\fn}{\mathfrak{n}}
\newcommand{\fm}{\mathfrak{m}}

\newcommand{\Ann}{\mathrm{Ann}}

\begin{document}

\title[]{Notes on the module of Euler systems}

\author{Ryotaro Sakamoto}

\begin{abstract} 
In this paper, we study the module of Euler systems. 
We determine the ideal of an Iwasawa algebra associated with Euler systems of rank $0$. 
We also show that  the module of higher rank Euler systems for $\bG_{m}$ over a totally real field is free of rank $1$ under the assumptions that Greenberg conjecture holds true and that the $\mu$-invariant of a certain Iwasawa module vanishes. 
\end{abstract}

\address{Center for Advanced Intelligence Project\\ RIKEN\\ Japan. } 
\email{ryotaro.sakamoto@riken.jp}

\subjclass[2020]{Primary: 11R23; Secondary: 11R34}
\keywords{Euler system, Iwasawa theory, Greenberg conjecture}

\maketitle

\setcounter{tocdepth}{1}

\tableofcontents

\section{Introduction}

The theory of Euler systems is a powerful tool for studying the relationship between Selmer groups and $L$-values.
While the importance of Euler systems is widely known, it seems that the structure of the module of Euler systems has not been well studied. 
In this paper, we study  the module of Euler systems.

\subsection{Euler systems of rank $0$}

To explain the  result we obtain in this paper, we introduce some notations. 
Throughout this paper,  $p$ denotes an odd prime and $k$ denotes a number field. 
We fix a finite set $S$ of places of $k$ satisfying $\{v \mid p\infty\} \subset S$. 
Let $\cO$ be a complete discrete valuation ring with residue characteristic $p$ and $T$ a free $\cO$-module of finite rank on which $G_{k,S} := \Gal(k_{S}/k)$ acts continuously, where $k_{S}$ is the maximal extension of $k$ which is unramified outside $S$. 
We also take a pro-$p$ abelian extension $\cK$ of $k$ satisfying $S_{\rm ram}(\cK/k) \cap S = \emptyset$. 
Let $\Omega$ denote the set of all finite extensions of $k$ in  $\cK$: 
\[
\Omega := \{K \mid k \subset K \subset \cK, \, [K \colon k] < \infty\}. 
\]
For any field $K \in \Omega$, we put 
\begin{align*}
\Lambda_{K} := \cO[[\Gal(k_{\infty}K/k)]] \,\,\, \textrm{ and } \,\,\,  \bT_{K} := \varprojlim_{n>0} \mathrm{Ind}_{G_{k}}^{G_{k_{n}K}}(T).  
\end{align*}
Here $k \subset k_{1} \subset k_{2} \subset \cdots $ is a sequence of finite extensions of $k$ with $\bigcup_{n>0}k_{n} = k_{\infty}$ 
and, for any field $L$, we denote by $G_{L} := \Gal(\overline{L}/L)$ the absolute Galois group of $L$.

\begin{definition} 
For any prime $\fq \not\in S$ of $k$, we define the Frobenius characteristic polynomial at $\fq$ by  
\[
P_{\fq}(x) := \det(1- x \cdot {\rm Fr}_{\fq} \mid T) \in \cO[x].
\]
We define the module ${\rm ES}_{0}(T)$ of Euler systems of rank $0$ by 
\[
{\rm ES}_{0}(T) := \left\{ (c_{K})_{K \in \Omega} \in \prod_{K \in \Omega}\Lambda_{K} \ \middle| \ 
\begin{array}{l}
\varphi_{K, L}(c_{K}) = \left(\prod_{\fq } P_{\fq}({\rm Fr}_{\fq}^{-1})\right) \cdot c_{L} 
\\
\textrm{ for any fields $K, L \in \Omega$ with $L \subset K$ } 
\end{array}
\right\}. 
\]
Here $\fq$ runs over the set $\mathrm{S}_{\rm ram}(K/k) \setminus \mathrm{S}_{\rm ram}(L/k)$ and $\varphi_{K, L} \colon \Lambda_{K} \longrightarrow \Lambda_{L}$ denotes the natural projection map. 
\end{definition}

In these settings, one of our main results are the following

\begin{theorem}[{Theorem \ref{thm:main1} and  Proposition \ref{prop:char-local}}]
For any filed $K \in \Omega$, we have 
\[
\langle c_{K} \mid (c_{L})_{L \in \Omega} \in {\rm ES}_{0}(T) \rangle_{\Lambda_{K}[1/p]} = \prod_{\fq \in S_{\rm ram}(K/k)} \mathrm{char}_{\Lambda_{K}}(H^{1}(G_{k_{\fq}}, \bT_{K})) \Lambda_{K}[1/p]. 
\]
Here $\mathrm{char}_{\Lambda_{K}}(H^{1}(G_{k_{\fq}}, \bT_{K})) \subset \Lambda_{K}$ denotes the characteristic ideal of $H^{1}(G_{k_{\fq}}, \bT_{K})$ defined in Definition \ref{def:char}. 

Moreover, we have $\langle c_{k} \mid (c_{L})_{L \in \Omega} \in {\rm ES}_{0}(T) \rangle = \Lambda_{k}$. 
\end{theorem}

\begin{remark}
For each integer $r \geq 1$, we write $\mathrm{ES}_{r}(T)$ for the modules of Euler systems of rank $r$ (see Definition~\ref{def:euler system}). 
One of the main results of the paper  \cite{MRkoly} by Mazur and Rubin, and the paper \cite{bss} by Burns, the author, and Sano shows that there is a natural homomorphism 
\[
\mathrm{ES}_{r}(T) \longrightarrow \mathrm{KS}_{r}(T). 
\]
Here $\mathrm{KS}_{r}(T)$ denotes  the modules of Kolyvagin systems of rank $r$ (see \cite[\S5.1]{bss}). 

Recently, in the paper \cite{koly0}, the author introduced the module $\mathrm{KS}_{0}(T)$ of Kolyvagin systems of rank $0$, based on the work of Kurihara in \cite{Kur12, Kur14a}. 
It is natural to expect that there exists a natural homomorphism 
\[
\mathrm{ES}_{0}(T) \longrightarrow \mathrm{KS}_{0}(T)
\]
 from previous studies. 
However,  Theorem \ref{thm:main1} shows that this expectation does not hold. 
In fact, if we have a natural homomorphism $\mathrm{ES}_{0}(T) \longrightarrow \mathrm{KS}_{0}(T)$, then the theory of rank $0$ Kolyvagin systems developed in \cite{koly0} shows that the ideal $\langle c_{k} \mid (c_{L})_{L \in \Omega} \in {\rm ES}_{0}(T) \rangle$ is contained in the initial Fitting ideal of an appropriate $\Lambda_{k}$-adic Selmer module, and hence Theorem \ref{thm:main1} implies that this ideal is equal to $\Lambda_{k}$. 
There are many examples where  this ideal is not trivial, so we conclude that we can not have a natural homomorphism $\mathrm{ES}_{0}(T) \longrightarrow \mathrm{KS}_{0}(T)$ in general. 

On the other hand, in the paper \cite{p-sel}, the author constructed  a rank $0$ Kolyvagin system from modular symbols, which is an Euler system of rank $0$. 
Hence the author expects that there is a subgroup $\cE(T) \subset \mathrm{ES}_{0}(T)$ with a natural homomorphism 
\[
\cE(T) \longrightarrow \mathrm{KS}_{0}(T). 
\]
\end{remark}

\subsection{The module of Euler systems for $\bG_{m}$}
Suppose that $k$ is a totally real field with $[k \colon \bQ] = r$. 
Let 
\[
\chi \colon \Gal(\overline{k}/k) \longrightarrow \overline{\bQ}^{\times}
\] 
be a non-trivial finite order even character and put $k_{\chi} := \overline{k}^{\ker(\chi)}$. 
Fix an embedding $\overline{\bQ} \longhookrightarrow \overline{\bQ}_{p}$. 
We also suppose that $p \nmid h_{k} \cdot [k_{\chi} \colon k]$, where $h_{k}$ denotes  the class number of $k$ and that 
\begin{align*}
S = \{v \mid p\infty \} \cup S_{\rm ram}(k_{\chi}/k), \,\,\, \cO = \bZ_{p}[\im(\chi)], \,\,\, \textrm{and} \,\,\, T = \cO(1) \otimes \chi^{-1}.
\end{align*} 
We write $M_{k_{\chi},\infty}$ for the maximal $p$-ramified pro-$p$ abelian extension of $k_{\chi}(\mu_{p^{\infty}})$. 
 We take $k_{\infty}/k$ to be the cyclotomic $\bZ_{p}$-extension. 
 Note that we have the canonical isomorphism 
 \[
 \Gal(k_{\chi}(\mu_{p^{\infty}})/k) \stackrel{\sim}{\longrightarrow} \Gal(k_{\infty}/k) \times \Gal(k_{\chi}(\mu_{p})/k) 
 \]
since  $p \nmid [k_{\chi} \colon k]$. 
We put  
\begin{align*}
X^{\chi}_{K} := e_{\chi}(\cO \otimes_{\bZ_{p}} \Gal(M_{k_{\chi},\infty}/k_{\chi}(\mu_{p^{\infty}}))), 
\end{align*}
where 
\[
e_{\chi} := [k_{\chi}(\mu_{p}) \colon k]^{-1} \sum_{\sigma \in \Gal(k_{\chi}(\mu_{p})/k)}\chi(\sigma)\sigma^{-1}. 
\]
%
Under this setting, we show the following

\begin{theorem}[Theorem~\ref{thm:main}]\label{main}
Let $\fm$ denote the maximal ideal of $\cO$. 
Suppose that 
\begin{itemize}
\item  $H^{0}(G_{k_{\fp}}, T/\fm T) = H^{2}(G_{k_{\fp}},T/\fm T) = 0$  for any prime $\fp$ of $k$ above $p$, 
\item the module $X_{k}^{\chi}$ is a finitely generated $\bZ_{p}$-module, 
\item Greenberg conjecture (Conjecture~\ref{conj:greenberg}) holds true. 
\end{itemize}
Then the $\cO[[\Gal(k_{\infty}\cK/k)]]$-module ${\rm ES}_{r}(T)$ is free of rank $1$ and is generated by an Euler system related to $p$-adic $L$-functions. 
\end{theorem}

When $k = \bQ$, Theorem~\ref{main} is closely related to a conjecture about the universality of the circular distribution  proposed by Coleman. 
This conjecture implies that any Euler system of rank $1$ for $\bG_{m}$ over $\bQ$ is essentially made out of cyclotomic units. 
The Coleman conjecture was studied by Seo in  \cite{Seo1, Seo2} and  by Burns and Seo in \cite{BS}. 
The authors obtained strong evidence in support of the Coleman conjecture. 
Although Theorem~\ref{main} in the case that $k = \bQ$ was essentially proved by Seo in \cite[Theorem~A]{Seo1},  the proof of the main result of \cite{Seo1} has an error as mentioned in \cite[Remark~3.6]{BS}. Some arguments in \cite{Seo1} can only be corrected either by assuming certain Galois descent property on distributions or by inverting certain primes. 

Recently, in \cite{BDSS}, Burns, Daoud, Sano, and Seo formulated a natural generalization of the conjecture of Coleman, which assert that,  modulo minor technical issues concerning torsion, an Euler system of an appropriate rank for $\bG_{m}$ over a number field is basic (see \cite[Conjecture~2.5]{BDSS} for the detail). 
In the present paper, we do not give the definition of basic Euler systems. 
However, it is worth mentioning that, by using Theorem~\ref{main}, we can give a new evidence in support of \cite[Conjecture~2.5]{BDSS} (see Theorem~\ref{thm:main2}).

\subsection{Notation}

For any field $k$, we fix a separable closure $\overline{k}$ of $k$ and 
denote by $G_{k} := \Gal(\overline{k}/k)$ the absolute Galois group of $k$.

For a profinite group $G$ and a topological $G$-module $M$, let $C^{\bullet}(G,M)$ denote the complex of inhomogeneous continuous cochains of $G$ with values in $M$. 
We also denote the object in the derived category corresponding to the complex $C^{\bullet}(G,M)$ by ${\bf R}\Gamma(G,M)$. 
For each integer $i \geq 0$, we write $H^{i}(G,M)$ for its $i$-th cohomology group. 

For any algebraic extension $k/\bQ$ and places $v$ of $\bQ$, we denote by $S_{v}(k)$ the set of places of $k$ above $v$.  
For any prime $\fq$ of $k$, we denote by $k_{\fq}$ the completion of $k$ at $\fq$. 
For any algebraic extension $K/k$, we denote by $S_{\rm ram}(K/k)$ the set of primes at which $K/k$ is ramified.

\subsection{Acknowledgments}

The author would like to express his gratitude to his supervisor Takeshi Tsuji for many helpful discussions. 
The author is also grateful to David Burns for helpful remarks on the conjecture of Coleman. 
The author was supported by JSPS KAKENHI Grant Number 20J00456.

\section{Definition of Euler systems}\label{sec:euler}

In this section, we recall the definition of  Euler systems.  The contents of this section are based on \cite[\S2]{hres}.

First, let us fix some notations. 
As in the introduction, let $k$ be a number field and 
we fix a finite set $S$ of places of $k$ satisfying $S_{\infty}(k) \cup S_{p}(k) \subset S$. 
Let $\cO$ be a complete discrete valuation ring with maximal ideal $\fm$ and $T$ a free $\cO$-module of finite rank on which $G_{k,S}$ acts continuously. 
Suppose that 
\begin{itemize}
\item  $p$ is coprime to the class number of $k$, 
\item the characteristic of $\cO/\fm$ is  $p$, and 
\item  $H^{0}(G_{k,S}, T/\fm T)$ vanishes. 
\end{itemize}
Let   $k_{\infty}/k$ be a $\bZ_{p}^{s}$-extension such that $s \geq 1$ and no prime of $k$ splits completely in  $k_{\infty}$. We put 
\[
\Lambda := \cO[[\Gal(k_{\infty}/k)]] \,\,\, \text{ and } \,\,\, \bT := \varprojlim_{n>0} \mathrm{Ind}_{G_{k}}^{G_{k_{n}}}(T). 
\]
Here  $k \subset k_{1} \subset k_{2} \subset \cdots $ are finite extensions of $k$ with $\bigcup_{n>0}k_{n} = k_{\infty}$. 
We also take a pro-$p$ abelian extension $\cK \subset \overline{k}$ of $k$ satisfying $S_{\rm ram}(\cK/k) \cap S = \emptyset$. 
Recall that 
\[
\Omega = \{K \mid k \subset K \subset \cK, \, [K \colon k] < \infty\}. 
\]
For each field $K \in \Omega$, we put 
\begin{itemize}
\item $S_{K} := S \cup S_{\rm ram}(K/k)$, 
\item $\Lambda_{K} := \cO[[\Gal(k_{\infty}K/k)]]$, 
\item $\bT_{K} :=  \varprojlim_{n>0} \mathrm{Ind}_{G_{k}}^{G_{k_{n}K}}(T)$.  
\end{itemize}
Since $p$ is coprime to the class number of $k$, we have  the canonical isomorphism 
\[
\Gal(k_{\infty}K/k) \stackrel{\sim}{\longrightarrow} \Gal(k_{\infty}/k) \times \Gal(K/k). 
\]
In this paper, by using this isomorphism, we identify $\Lambda_{K}$ with the group ring $\Lambda[\Gal(K/k)]$.


\begin{definition} 
For each field $K \in \Omega$,  let $M_{K}$ be a $\Lambda_{K}$-module. 
Suppose that $\{M_{K}\}_{K \in \Omega}$ is an inverse system of $\cO[[\Gal(k_{\infty}\cK/k)]]$-modules with transition maps  $\varphi_{K, L} \colon M_{K} \longrightarrow M_{L}$ for any $K, L \in \Omega$ with $L \subset K$. 
We then define a module ${\rm ES}(\{M_{K}\}_{K\in\Omega})$ to be  
\[
\left\{ (m_{K})_{K \in \Omega} \in \prod_{K \in \Omega}M_{K} \ \middle| \ 
\begin{array}{l}
\varphi_{K, L}(m_{K}) = \left(\prod_{\fq \in S_{K} \setminus S_{L}} P_{\fq}({\rm Fr}_{\fq}^{-1})\right) \cdot m_{L} 
\\
\textrm{ for any fields $K, L \in \Omega$ with $L \subset K$ } 
\end{array}
\right\}. 
\]
\end{definition}

\begin{definition}
For any commutative ring $R$ and any $R$-module $M$, 
we put 
\[
M^{*} := \Hom_{R}(M,R). 
\]  
For any integer $r \geq 0$, we define an $r$-th exterior bi-dual ${\bigcap}^{r}_{R}M$ of $M$ by  
\[
{\bigcap}^{r}_{R}M := \left({\bigwedge}^{r}_{R}(M^{*})\right)^{*}. 
\] 
\end{definition}

\begin{lemma}[{\cite[Lemma~2.5]{hres}}]\label{lem:inverse-system}
Let $r \geq 0$ be an integer. 
For any fields $K, L \in \Omega$ with $L \subset K$, the canonical homomorphism $\bT_{K} \longrightarrow \bT_{L}$ induces an $\cO[[\Gal(k_{\infty}\cK/k)]]$-homomorphism 
\[
{\bigcap}^{r}_{\Lambda_{K}}H^{1}(G_{k,S_{K}},\bT_{K}) \longrightarrow {\bigcap}^{r}_{\Lambda_{L}}H^{1}(G_{k,S_{L}},\bT_{L}). 
\]
Hence we have an inverse system of $\cO[[\Gal(k_{\infty}\cK/k)]]$-modules 
\[
\left\{{\bigcap}^{r}_{\Lambda_{K}}H^{1}(G_{k,S_{K}},\bT_{K})\right\}_{K\in \Omega}. 
\] 
\end{lemma}

\begin{definition}\label{def:euler system}
For any integer $r \geq 0$, 
we define the module ${\rm ES}_{r}(T)$ of Euler systems of rank $r$ (for $T$) by  
\[
{\rm ES}_{r}(T) := {\rm ES}\left(\left\{{\bigcap}^{r}_{\Lambda_{K}}H^{1}(G_{k,S_{K}},\bT_{K})\right\}_{K \in \Omega}\right). 
\]
\end{definition}

\section{Euler systems of rank $0$}

We use the same notation as in \S\ref{sec:euler}. 
Let $K, L \in \Omega$ be fields with $L \subset K$. 
To simplify the notation, we  set 
\begin{itemize}
\item $\cG_{K} := \Hom(\Gal(K/k), \overline{\bQ}^{\times})$,  and 
\item $\Lambda_{K, \overline{\bQ}_{p}} := \Lambda_{K} \otimes_{\cO}\overline{\bQ}_{p}$. 
\end{itemize}
We then have the canonical injection  $\cG_{L} \longhookrightarrow \cG_{K}$. 
Hence we identify $\cG_{L}$ with the subgroup $\{\psi \in \cG_{K} \mid \psi(\Gal(K/L)) = 1\}$ of $\cG_{K}$ by using this injection. 
For each character $\psi \in \cG_{K}$, we set 
\[
e_{K,\psi} := \frac{1}{[K \colon k]} \sum_{g \in \Gal(K/k)}\psi(g)g^{-1} \in \overline{\bQ}_{p}[\Gal(K/k)]. 
\]
We note that 
\[
\Lambda_{K, \overline{\bQ}_{p}}=  \prod_{\psi \in \cG_{K}}\Lambda_{K, \overline{\bQ}_{p}}e_{K,\psi}  
\]
and $\Lambda_{K, \overline{\bQ}_{p}}e_{K,\psi} \cong \Lambda \otimes_{\cO} \overline{\bQ}_{p}$ is a principal ideal domain for any character $\psi \in \cG_{K}$. 
We write  $\pi_{K,L} \colon \Lambda_{K, \overline{\bQ}_{p}} \longrightarrow \Lambda_{L, \overline{\bQ}_{p}}$ for the canonical projection. 
For any character $\psi \in \cG_{L}$, we have 
\[
\pi_{K,L}(e_{K,\psi}) = e_{L, \psi}. 
\]
Hence the homomorphism $\pi_{K,L}$ induces an isomorphism 
\[
\Lambda_{K, \overline{\bQ}_{p}}e_{K,\psi} \stackrel{\sim}{\longrightarrow} \Lambda_{L, \overline{\bQ}_{p}}e_{L,\psi}. 
\]

Let $\fq \not\in S$ be a prime of $k$. 
We write $\cI_{K,\fq}$ for the inertia subgroup of  $\Gal(k_{\infty}K/k)$ at $\fq$. We then have the arithmetic Frobenius element ${\rm Fr}_{\fq} \in \Gal(k_{\infty}K/k)/\cI_{K,\fq}$. 
Since $k_{\infty}/k$ is unramified at $\fq \nmid p$, 
the canonical homomorphism $\cI_{K,\fq} \longrightarrow \Gal(K/k)$ is injective. 
Hence we can identify $\cI_{K,\fq}$ with the inertia subgroup of $\Gal(K/k)$ at $\fq$.


\begin{definition}
Let $\fq \not\in S$ be a prime of $k$.  
We define an element $f_{K,\fq} \in \Lambda_{K, \overline{\bQ}_{p}}$ by 
\[
f_{K,\fq}e_{K,\psi} := 
\begin{cases}
e_{K,\psi} & \textrm{ if } \,\,\, \psi(\cI_{K,\fq}) \neq 1, 
\\
P_{\fq}({\rm Fr}_{\fq}^{-1})e_{K,\psi} & \textrm{ if } \,\,\, \psi(\cI_{K,\fq}) = 1 
\end{cases}
\]
for each character $\psi \in \cG_{K}$. 
We put 
\[
f_{K} := \prod_{\fq \in S_{\rm ram}(K/k)}f_{K,\fq}. 
\]
Note that $f_{K} \in \Lambda_{K}[1/p]$ and $f_{K}$ is a regular element of $\Lambda_{K}[1/p]$. 
\end{definition}

\begin{lemma}\label{lemma:euler-rel}
Let  $K, L \in \Omega$ be fields satisfying $L \subset K$. 
\begin{itemize}
\item[(i)] For any prime $\fq \not\in S$ of $k$, we have 
\[
\pi_{K, L}(f_{K,\fq}) = f_{L,\fq}.  
\]
\item[(ii)]
We have 
\[
\pi_{K, L}(f_{K}) = \left(\prod_{\fq \in S_{\rm ram}(K/k) \setminus S_{\rm ram}(L/k)}P_{\fq}({\rm Fr}_{\fq}^{-1})\right)f_{L}. 
\]
Hence $(f_{K})_{K \in \Omega} \in  {\rm ES}(\{\Lambda_{K}[1/p]\}_{K \in \Omega})$. 
\end{itemize}
\end{lemma}
\begin{proof}
Let us show the claim (i). 
Let $\fq \not\in S$ be a prime. It suffices to show that 
\[
\pi_{K,L}(f_{K,\fq})e_{L,\psi} = f_{L,\fq}e_{L,\psi} 
\]
for each character $\psi \in \cG_{L} \subset \cG_{K}$. 
Since $\psi \in \cG_{L}$, we have $\psi(\Gal(K/L)) = 1$, and hence $\psi(\cI_{K,\fq}) = 1$ if and only if  $\psi(\cI_{L,\fq}) = 1$. 
Since $\pi_{K,L}(e_{K,\psi}) = e_{L,\psi}$, the definitions of $f_{K,\fq}$ and $f_{L,\fq}$ shows that 
$\pi_{K,L}(f_{K,\fq})e_{L,\psi} = \pi_{K,L}(f_{K,\fq}e_{K,\psi}) = f_{L,\fq}e_{L,\psi}$. 

Since $f_{L,\fq} = P_{\fq}(\mathrm{Fr}_{\fq}^{-1})$ for any prime $\fq \not\in S_{L}$ of $k$, claim (ii) immediately follows  from claim (i). 
\end{proof}

\begin{proposition}\label{prop:key_div}
For any system $(c_{K})_{K \in \Omega} \in  {\rm ES}(\{\Lambda_{K}[1/p]\}_{K \in \Omega})$, we have 
\[
c_{K} \in f_{K}  \Lambda_{K}[1/p]
\]
for any field $K \in \Omega$. 
\end{proposition}
\begin{proof}
Take a field $K \in \Omega$. 
Let us prove this proposition by induction on $d := \#S_{\rm ram}(K/k)$. 
When $d = 0$, we have $f_{K} = 1$, and hence $c_{K} \in f_{K}  \Lambda_{K}[1/p]$. 

Suppose that $d > 0$. 
Take a character $\psi \in \cG_{K}$. 
Since the $\Lambda_{K}[1/p]$-algebra $\Lambda_{K, \overline{\bQ}_{p}}$ is faithfully flat, 
we only need to show that 
\begin{align}\label{claim1}
c_{K}e_{K,\psi} \in f_{K} \Lambda_{K, \overline{\bQ}_{p}} e_{K,\psi}. 
\end{align}
If $\psi(\cI_{K,\fq}) \neq 1$ for any prime $\fq \in S_{\rm ram}(K/k)$, then we have $f_{K}e_{K,\psi} = e_{K,\psi}$, and hence $c_{K}e_{K,\psi} \in f_{K}  \Lambda_{K, \overline{\bQ}_{p}} e_{K,\psi}$. 

Suppose that there is a prime $\fq \in S_{\rm ram}(K/k)$ with $\psi(\cI_{K,\fq}) = 1$. 
Put $L := K^{\cI_{K,\fq}}$. 
Note that $\psi \in \cG_{L}$. 
Since $\pi_{K,L}$ induces an isomorphism 
$\Lambda_{K, \overline{\bQ}_{p}}e_{K,\psi} \stackrel{\sim}{\longrightarrow} \Lambda_{L, \overline{\bQ}_{p}}e_{L,\psi}$, 
the claim~\eqref{claim1} is equivalent to that 
\[
\pi_{K, L}(c_{K})e_{L, \psi} \in \pi_{K,L}(f_{K})  \Lambda_{L,\overline{\bQ}_{p}}e_{L, \psi}. 
\]
The definition of Euler systems shows that 
\[
\pi_{K, L}(c_{K}) = \left(\prod_{\fq \in S_{\rm ram}(K/k) \setminus S_{\rm ram}(L/k)}P_{\fq}({\rm Fr}_{\fq}^{-1}) \right) c_{L}. 
\]
Furthermore, since $\fq \not\in S_{\rm ram}(L/k)$, 
we have $\# S_{\rm ram}(L/k) \leq  d - 1$. 
Hence the induction hypothesis shows $c_{L} \in f_{L} \Lambda_{L}[1/p]$. 
Therefore, by Lemma~\ref{lemma:euler-rel}, we conclude that 
\begin{align*}
\pi_{K, L}(c_{K})e_{L, \psi} 
&= \left(\prod_{\fq \in S_{\rm ram}(K/k) \setminus S_{\rm ram}(L/k)}P_{\fq}({\rm Fr}_{\fq}^{-1}) \right) c_{L} e_{L,\psi}
\\
&\in \left(\prod_{\fq \in S_{\rm ram}(K/k) \setminus S_{\rm ram}(L/k)}P_{\fq}({\rm Fr}_{\fq}^{-1}) \right) f_{L} \Lambda_{L, \overline{\bQ}_{p}} e_{L, \psi}
\\
&= \pi_{K,L}(f_{K})  \Lambda_{L, \overline{\bQ}_{p}} e_{L, \psi}. 
\end{align*}
\end{proof}

\begin{corollary}
The homomorphism 
\begin{align}\label{hom}
\varprojlim_{K \in \Omega} \Lambda_{K}[1/p] \longrightarrow {\rm ES}(\{\Lambda_{K}[1/p]\}_{K \in \Omega}); \, (\lambda_{K})_{K \in \Omega} \mapsto (\lambda_{K}f_{K})_{K\in \Omega}
\end{align}
is an isomorphism. 
\end{corollary}
\begin{proof}
By Lemma~\ref{lemma:euler-rel}, the homomorphism~\eqref{hom} is well-defined. 
Since $f_{K}$ is a regular element of $\Lambda_{K}[1/p]$ for any field $K \in \Omega$, the homomorphism~\eqref{hom} is injective. Furthermore, Proposition~\ref{prop:key_div} shows that the homomorphism~\eqref{hom} is surjective. 
\end{proof}

\begin{theorem}\label{thm:main1}
For each field $K \in \Omega$, we define an ideal $I_{K} \subset \Lambda_{K}$ by 
\[
I_{K} := \{c_{K} \mid (c_{L})_{L \in \Omega}\in \mathrm{ES}_{0}(T) \}. 
\]
Then the ideal $I_{K} \Lambda_{K}[1/p]$ is generated by $f_{K}$ and we have $I_{k} = \Lambda$. 
\end{theorem}
\begin{proof}
Proposition \ref{prop:key_div} shows that $I_{K} \Lambda_{K}[1/p] \subset f_{K}\Lambda_{K}[1/p]$. 
Let us show the opposite inclusion. 
For each prime $\fq$ of $k$, we take a lift of arithmetic Frobenius element $\widetilde{\mathrm{Fr}}_{\fq} \in \Gal(k_{\infty}\cK/k)$. 
Then we have an Euler system $(c_{K})_{K \in \Omega}$ defined by 
\[
c_{K} := \prod_{\fq \in S_{\rm ram}(K/k)}P_{\fq}(\widetilde{\mathrm{Fr}}_{\fq}^{-1}) \in \Lambda_{K}, 
\]
and hence $c_{K} \in I_{K}$. 
In particular, $I_{k} = \Lambda$ since $c_{k} = 1$.

Let $\cI_{\fq} \subset \Gal(k_{\infty}\cK/k)$ denote the inertia subgroup at $\fq$. 
Take an ideal $\fn$ of $k$ and  $\sigma_{\fq} \in \cI_{\fq}$. 
Consider the system $(c(\fn, \{\sigma_{\fq}\}_{\fq} )_{K})_{K\in\Omega}$ defined by 
\[
c(\fn, \{\sigma_{\fq}\}_{\fq} )_{K} := 
\begin{cases}
0 & \textrm{ if } \{\fq \mid \fn\} \not\subset S_{\rm ram}(K/k), 
\\
\prod_{\fq \mid \fn} \sigma_{\fq} \times  \prod_{\substack{\fq' \in S_{\rm ram}(K/k) \\ \fq' \nmid \fn }}P_{\fq'}(\widetilde{\mathrm{Fr}}_{\fq'}^{-1}) & \textrm{ if } \{\fq \mid \fn\} \subset  S_{\rm ram}(K/k). 
\end{cases}
\]
It is easy to check that $(c(\fn, \{\sigma_{\fq}\}_{\fq} )_{K})_{K\in\Omega}$ is an Euler system of rank $0$. 
Therefore, we conclude that 
\[
\prod_{\fq \in S_{\rm ram}(K/k)}(P_{\fq}(\widetilde{\mathrm{Fr}}_{\fq}^{-1})+ \Lambda[\cI_{K,\fq}]) \subset I_{K}
\]
since $\sigma_{\fq}$ is an arbitrary element of $\cI_{\fq}$. 
Take a non-trivial character $\psi \colon \Gal(K/k) \longrightarrow \overline{\bQ}^{\times}$. 
If $\psi(\cI_{K,\fq}) =1$, then 
\[
(P_{\fq}(\widetilde{\mathrm{Fr}}_{\fq}^{-1})+ \Lambda[\cI_{K,\fq}])e_{K,\psi} = P_{\fq}({\mathrm{Fr}}_{\fq}^{-1})e_{K,\psi} = f_{K,\fq}e_{K,\psi}
\] 
by definition. Moreover, if $\psi(\cI_{K,\fq}) \neq 1$, then we have 
\[
\Lambda_{K, \overline{\bQ}_{p}}(P_{\fq}(\widetilde{\mathrm{Fr}}_{\fq}^{-1})+ \Lambda[\cI_{K,\fq}])e_{K,\psi} = 
\Lambda_{K, \overline{\bQ}_{p}}e_{K,\psi}. 
\]
Since $\Lambda_{K, \overline{\bQ}_{p}}$ is a faithfully flat $\Lambda_{K}[1/p]$-algebra, 
these facts imply that 
\[
f_{K}\Lambda_{K}[1/p] \subset I_{K} \Lambda_{K}[1/p], 
\]
and hence $f_{K}\Lambda_{K}[1/p] = I_{K} \Lambda_{K}[1/p]$. 
\end{proof}

\section{Characteristic ideals}

In this section, we recall the definition of the characteristic ideal of a finitely generated module over a noetherian ring. 
 Basic properties of  the characteristic ideal  studied in \cite[Appendix~C]{hres}. 

\begin{definition}[{\cite[Definition~2.8]{hres}}]\label{def:char} 
Let $R$ be a noetherian ring and $M$ a finitely generated $R$-module. 
Take an exact sequence $0 \longrightarrow N \longrightarrow R^{s} \longrightarrow M \longrightarrow 0$ of $R$-modules with $s \geq 1$. 
We then define the characteristic ideal of $M$ by 
\[
{\rm char}_{R}(M) := {\rm im }\left({\bigcap}_{R}^{s}N \longrightarrow {\bigcap}_{R}^{s}R^{s} = R\right). 
\]
We see that the characteristic ideal ${\rm char}_{R}(M)$ does not depend on the choice of the exact sequence $0 \longrightarrow N \longrightarrow R^{s} \longrightarrow M \longrightarrow 0$  of $R$-modules (see \cite[Remark~C.5]{hres}). 
\end{definition}

\begin{remark}
Since the exterior bi-dual commutes with flat base change, for a flat homomorphism $R \longrightarrow R'$ of noetherian rings and a finitely generated $R$-module $M$, we have ${\rm char}_{R}(M)R' = {\rm char}_{R'}(M \otimes_{R} R')$. 
\end{remark}

\begin{remark}
The characteristic ideal is not additive in  short exact sequences. 
In fact, suppose that $R$ is a noetherian local ring. 
Take a regular element $r \in R$ and an ideal $I$ of $R$ with $r \in I$. 
We then have an exact sequence of $R$-modules
\[
0 \longrightarrow I/rR \longrightarrow R/rR \longrightarrow R/I \longrightarrow 0. 
\]
We note that ${\rm char}_{R}(R/rR) = rR \cong R$ since $r$ is a regular element. 
Therefore, if we have 
\[
{\rm char}_{R}(R/I){\rm char}_{R}(I/rR) = {\rm char}_{R}(R/rR) \cong R, 
\]
the ideal 
${\rm char}_{R}(R/I)$ is invertible. 
Since an invertible ideal is projective, the ideal  ${\rm char}_{R}(R/I)$ is principal. 
Hence we conclude that if ${\rm char}_{R}(R/I) = \mathrm{im}(I^{**} \longrightarrow R)$ is not principal, then we have 
\[
{\rm char}_{R}(R/rR) \neq {\rm char}_{R}(R/I){\rm char}_{R}(I/rR). 
\]
\end{remark}

\begin{lemma}[{\cite[Proposition~C.7]{hres}}]\label{lemma:fitt}
Let $R$ be a noetherian ring and $M$ a finitely generated $R$-module.  
\begin{itemize}
\item[(i)] We have ${\rm Fitt}_{R}^{0}(M) \subset {\rm char}_{R}(M)$. 
\item[(ii)] If the projective dimension of $M$ is at most 1, then we have ${\rm Fitt}_{R}^{0}(M) = {\rm char}_{R}(M)$. 
\end{itemize}
\end{lemma}

When $R$ is a zero dimensional Gorenstein ring, the characteristic ideal of an $R$-module coincides with its annihilator ideal. 

\begin{proposition}\label{prop:artin-char}
Let  $R$ be a zero dimensional Gorenstein ring and $M$ a finitely generated $R$-module.  
Then we have  
\[
{\rm char}_{R}(M) = \Ann_{R}(M). 
\] 
\end{proposition}
\begin{proof}
Take an integer $r > 0$ and an exact sequence of $R$-modules 
\[
0 \longrightarrow N \longrightarrow R^{r} \longrightarrow M \longrightarrow 0. 
\]
Since the functor $(-)^{*}$ is exact, we have  an exact sequence 
\[
M^{*} \otimes_{R} {\bigwedge}^{r-1}_{R}(R^{r})^{*} \longrightarrow {\bigwedge}^{r}_{R}(R^{r})^{*} 
\longrightarrow {\bigwedge}^{r}_{R}N^{*} \longrightarrow 0
\]
(see \cite[Lemma 2.5]{bss}). 
Note that $M = M^{**}$ by Matlis duality. 
By taking $R$-duals to this exact sequence, we obtain an exact sequence 
\[
0 \longrightarrow {\bigcap}^{r}_{R}N \longrightarrow R \longrightarrow 
\left(M^{*} \otimes_{R} {\bigwedge}^{r-1}_{R}(R^{r})^{*} \right)^{*} = 
M^{r}. 
\]
The homomorphism $R \longrightarrow M^{r}$ corresponds to the surjection $R^{r} \longrightarrow M$ 
under the canonical identifications $\Hom_{R}(R,M^{r}) = \bigoplus^{r}_{i=1} \Hom_{R}(R,M) = \Hom_{R}(R^{r},M)$. 
Hence we have  
\begin{align*}
\mathrm{char}_{R}(M) = \ker(R \longrightarrow M^{r}) 
= \Ann_{R}\left(\im(R \longrightarrow M^{r})\right). 
\end{align*}
Since the homomorphism $R \longrightarrow M^{r}$ corresponds to the surjection $R^{r} \longrightarrow M$, we conclude that 
$\mathrm{char}_{R}(M) =  \Ann_{R}\left(\im(R \longrightarrow M^{r})\right) =  \Ann_{R}(M)$. 
\end{proof}

\begin{definition}
Let $R$ be a noetherian ring and $n \geq 0$ an integer. 
\begin{itemize}
\item[(i)] The ring $R$ is said to satisfy the condition (G$_{n}$) if the local ring $R_{\fr}$ is Gorenstein for any  prime $\fr$ of $R$ with ${\rm ht}(\fr) \leq n$. 
\item[(ii)] We say that the ring $R$ satisfies  Serre's condition (S$_{n}$) if the inequality 
\[
{\rm depth}(R_{\fr}) \geq \min\{n, {\rm ht}(\fr)\}
\] 
holds for all prime ideal $\fr$ of $R$. 
\end{itemize}
\end{definition}

\begin{remark}
The ring $\Lambda_{K}$ satisfies (G$_{n}$) and (S$_{n}$) for any integer $n \geq 0$ since $\Lambda_{K} = \Lambda[\Gal(K/k)]$ is Gorenstein. 
\end{remark}

\begin{lemma}[{\cite[Lemma~C.1]{hres}}]\label{lemma:bidual-inj}
Let $R$ be a noetherian ring satisfying (G$_{0}$) and (S$_{1}$).  
For any integer $r \geq 0$ and any injection $M \longhookrightarrow N$ of finitely generated $R$-modules, 
the induced homomorphism ${\bigcap}^{r}_{R}M \longrightarrow {\bigcap}^{r}_{R}N$ is also injective. 
\end{lemma}

Recall that $(-)^{*} := \Hom_{R}(-,R)$ denotes the $R$-dual functor. 
If a noetherian ring  $R$ satisfies  (G$_{0}$) and (S$_{1}$),  the canonical homomorphism $I^{**} \longrightarrow R^{**} = R$ is injective for any ideal $I$ of $R$ by Lemma~\ref{lemma:bidual-inj}. 
Hence by identifying $I^{**}$ with the image of this injection, we regard $I^{**}$ as an ideal of $R$ when $R$ satisfies (G$_{0}$) and (S$_{1}$). 

\begin{lemma}[{\cite[Proposition~C.10]{hres}}]\label{lemma:ref}
Let $R$ be a noetherian ring satisfying (G$_{0}$) and (S$_{1}$). Let $M$ be a finitely generated $R$-module. 
Then the ideal ${\rm char}_{R}(M)$ is reflexive, that is, ${\rm char}_{R}(M) = {\rm char}_{R}(M)^{**}$. 
\end{lemma}

\begin{lemma}[{\cite[Lemma~C.11]{hres}}]\label{lemma:bidualeq}
Let $R$ be a noetherian ring satisfying (G$_{0}$) and (S$_{2}$). 
Let $I$ and $J$ be ideals of $R$. If $IR_{\fr} \subset JR_{\fr}$ for any prime ideal $\fr$ of $R$ with ${\rm ht}(\fr) \leq 1$, then  $I^{**} \subset J^{**}$. 
\end{lemma}

\begin{corollary}[{\cite[Remark~C.12~(i)]{hres}}]\label{cor:pseudo}
Let $R$ be a noetherian ring satisfying (G$_{0}$) and (S$_{2}$). 
 If a finitely generated $R$-module $M$ is pseudo-isomorphic to an $R$-module $N$, then we have ${\rm char}_{R}(M) = {\rm char}_{R}(N)$. 
\end{corollary}
\begin{proof}
The characteristic ideals of $M$ and $N$ are reflexive by Lemma~\ref{lemma:ref}. 
Hence by Lemma~\ref{lemma:bidualeq},  we may assume that the Krull dimension of $R$ is at most $1$.  
In this case, any pseudo-null module is zero, and hence $M$ is isomorphic to $N$, which implies  ${\rm char}_{R}(M) = {\rm char}_{R}(N)$.  
\end{proof}

\begin{corollary}[{\cite[Remark~C.12~(ii)]{hres}}]\label{cor:char-fitt}
Let $R$ be a normal ring and $M$ a finitely generated $R$-module. 
Then we have 
\[
{\rm char}_{R}(M) = {\rm Fitt}_{R}^{0}(M)^{**}. 
\]
Hence, in this case, the notion of the characteristic ideal coincides with the usual one. 
\end{corollary}
\begin{proof}
Lemma~\ref{lemma:ref} shows that ${\rm char}_{R}(M)$ is reflexive. Hence by Lemma~\ref{lemma:bidualeq},  we may assume that the Krull dimension of $R$ is at most $1$.  
Then $R$ is a discrete valuation ring, and the projective dimension of $M$ is at most $1$. 
Lemma~\ref{lemma:fitt}~(ii) shows that ${\rm Fitt}_{R}^{0}(M) = {\rm char}_{R}(M)$. 
\end{proof}

\begin{lemma}\label{lemma:torsion}
Let $R$ be a noetherian ring satisfying (G$_{0}$) and (S$_{2}$) and $M$ a finitely generated torsion $R$-module. 
Let $r \in R$ be a regular element. If $r$ is $M$-regular, then we have 
\[
{\rm char}_{R}(M) \cap rR = r \cdot {\rm char}_{R}(M). 
\]  
Furthermore, for any prime ideal $\fr$ of $R$ with ${\rm ht}(\fr) \leq 1$ and $r \in \fr$, the module $M \otimes_{R} R_{\fr}$ vanishes. 
\end{lemma}
\begin{proof}
Take an element $x \in R$ with $y := rx \in {\rm char}_{R}(M)$. Let us show $x \in {\rm char}_{R}(M)$. 
By Lemma~\ref{lemma:ref}, the ideal ${\rm char}_{R}(M)$ is reflexive. Since $R$ satisfies (S$_{2}$) and localization of modules is an exact functor, by Lemma~\ref{lemma:bidualeq}, we may assume that $R$ is a Cohen-Macaulay local ring with $\dim(R) \leq 1$.  
Furthermore, we may also assume that $r \in \fm_{R}$. Here $\fm_{R}$ denotes the maximal ideal of $R$.  

Put $I := \Ann_{R}(M)$. Let us show that $I = R$. We assume the contradiction, namely, $I \subset \fm_{R}$. 
Since $r$ is $M$-regular, the homomorphism 
\[
\End_{R}(M) \longrightarrow \End_{R}(M); \, f \mapsto rf
\] 
is injective, and $r$ is also $R/I$-regular since the homomorphism $R/I  \longrightarrow \End_{R}(M)$ is injective. 
This fact implies that 
\[
\dim(R/I) = \dim(R/(I + rR)) + 1. 
\]
Furthermore, since $R$ is Cohen-Macaulay, we have 
\[
\dim(R) = {\rm ht}(I) + \dim(R/I) = {\rm ht}(I) +  \dim(R/(I + rR)) + 1. 
\]
The fact that $M$ is a torsion $R$-module implies that ${\rm ht}(I) = {\rm grade}(I) \geq 1$, and hence $\dim(R)  \geq 2$. This contradicts the assmption that $\dim(R) \leq 1$. 

Since $I = R$, the module $M$ vanishes, which implies $x \in R =  {\rm char}_{R}(0) = {\rm char}_{R}(M)$. 
%
\end{proof}

\begin{proposition}\label{prop:char-local}
For any field $K \in \Omega$ and  prime $\fq \not\in S$ of $k$, we have 
\[
{\rm char}_{\Lambda_{K}}(H^{1}(G_{k_{\fq}}, \bT_{K}))\Lambda_{K}[1/p] = f_{K,\fq} \Lambda_{K}[1/p]. 
\] 
\end{proposition}
\begin{proof}
Put $\bV_{K} := \bT_{K} \otimes_{\cO} \overline{\bQ}_{p}$ and $H^{1}(G_{k_{\fq}}, \bV_{K}) := H^{1}(G_{k_{\fq}}, \bT_{K}) \otimes_{\cO} \overline{\bQ}_{p}$. 
Since the ring homomorphism $\Lambda_{K}[1/p] \longhookrightarrow \Lambda_{K, \overline{\bQ}_{p}}$ is faithfully flat, 
it suffice to prove that 
\[
{\rm char}_{\Lambda_{K, \overline{\bQ}_{p}}}(H^{1}(G_{k_{\fq}}, \bV_{K}))e_{K,\psi} = f_{K,\fq}  \Lambda_{K, \overline{\bQ}_{p}}e_{K,\psi}. 
\]
for any character $\psi \in \cG_{K}$. 

By \cite[Corollary~B.3.6]{R}, we have an isomorphism 
\begin{align}\label{isom:unr}
\bV_{K}^{\cI_{K,\fq}}/(1-{\rm Fr}_{\fq})\bV_{K}^{\cI_{K,\fq}} \cong H^{1}(G_{k_{\fq}}, \bV_{K}). 
\end{align}
Let $\psi \in \cG_{K}$ be a character. When $\psi(\cI_{K,\fq}) \neq 1$, the isomorphism~\eqref{isom:unr} shows that  the module 
$H^{1}(G_{k_{\fq}}, \bV_{K})e_{K,\psi}$ vanishes. 
Hence we have 
\[
{\rm char}_{\Lambda_{K, \overline{\bQ}_{p}}}(H^{1}(G_{k_{\fq}}, \bV_{K}))e_{K,\psi} = \Lambda_{K, \overline{\bQ}_{p}}e_{K,\psi} =  f_{K,\fq}  \Lambda_{K, \overline{\bQ}_{p}}e_{K,\psi}. 
\]
If $\psi(\cI_{K,\fq}) = 1$, the isomorphism~\eqref{isom:unr} shows that we have an exact sequence 
\[
0 \longrightarrow \bV_{K}e_{K,\psi} \xrightarrow{1-{\rm Fr}_{\fq}} \bV_{K}e_{K,\psi} \longrightarrow H^{1}(G_{k_{\fq}}, \bV_{K})e_{K,\psi} \longrightarrow 0. 
\]
Since $\bV_{K}e_{K,\psi}$ is a free $\Lambda_{K, \overline{\bQ}_{p}}e_{K,\psi}$-module, we have 
\[
{\rm char}_{\Lambda_{K, \overline{\bQ}_{p}}}(H^{1}(G_{k_{\fq}}, \bV_{K}))e_{K,\psi} = \det(1-{\rm Fr}_{\fq} \mid \bV_{K}e_{K,\psi}). 
\]
Furthermore, it is well-known that $\bT_{K} \cong T \otimes_{\cO} \Lambda_{K}^{\iota}$, where $\Lambda_{K}^{\iota}$ is a free $\Lambda_{K}$-module of rank $1$
defined by the homomorphism $\cO[[\Gal(k_{\infty}K/k)]] \longrightarrow \cO[[\Gal(k_{\infty}K/k)]]; g \mapsto g^{-1}$. 
Hence we conclude that 
\[
\det(1-{\rm Fr}_{\fq} \mid \bV_{K}e_{K,\psi}) = P_{\fq}({\rm Fr}_{\fq}^{-1})e_{K,\psi}.  
\]
\end{proof}

\section{On the module of  Euler systems for $\bG_{m}$}\label{sec:euler G_{m}} 


Suppose that $k$ is a totally real field with $[k \colon \bQ] = r$. Let 
\[
\chi \colon G_{k} \longrightarrow \overline{\bQ}^{\times}
\] 
be a non-trivial finite order even character.  
Put $k_{\chi} := \overline{k}^{\ker(\chi)}$. 
Assume that $p$ is coprime to the class number of $k$ and $[k_{\chi} \colon k]$. 
We put 
\[
\cO := \bZ_{p}[\im(\chi)] \, \text{ and } \, S := S_{\infty}(k) \cup S_{p}(k) \cup S_{\rm ram}(k_{\chi} /k). 
\]
Let 
\[
T := \cO(1) \otimes \chi^{-1}, 
\]
that is, $T \cong \cO$ as $\cO$-modules and the Galois group $G_{k,S}$ acts on $T$ via the character 
$\chi_{\rm cyc}\chi^{-1}$, where $\chi_{\rm cyc}$ denotes the cyclotomic character of $k$. 
We write $k_{\infty}/k$ for the cyclotomic $\bZ_{p}$-extension of $k$. 
Let $\cK$ denote the maximal pro-$p$ abelian extension of $k$ satisfying $S_{\rm ram}(\cK/k) \cap S = \emptyset$.  
Let $\Omega$, $\Lambda_{K}$, and $\bT_{K}$ be as in \S\ref{sec:euler}. 
For any field $K \in \Omega$ and any integer $i \geq 0$, we put 
\[
H^{i}(G_{k_{p}}, \bT_{K}) := \bigoplus_{\fp \in S_{p}(k)}H^{i}(G_{k_{\fp}},\bT_{K}). 
\]


\subsection{Stickelberger elements}
In this subsection, we will recall the definition of Stickelberger elements. 
The contents of this subsection are based on \cite[\S3.1]{hres}.

Let $K \in \Omega$ be a field and $n \geq 0$ an integer. 
Let $\mu_{p^{n}}$ denote the set of $p^{n}$-th roots of unity in $\overline{\bQ}$ and $\mu_{p^{\infty}} := \bigcup_{n>0}\mu_{p^{n}}$. 
For notational simplicity, we set
\[
G_{K,n} := \Gal(k_{\chi}K(\mu_{p^{n}})/k) = \Gal(k_{\chi}/k)  \times \Gal(K/k) \times \Gal(k(\mu_{p^{n}})/k). 
\] 
Let 
\[
\omega \colon G_{k,1} \longrightarrow \Gal(k(\mu_{p})/k) \longhookrightarrow \bZ_{p}^{\times} 
\] 
denote the Teichm\"ullar character. 
We write $\zeta_{k_{\chi}K,n,S}(s,\sigma)$ for the partial zeta function for $\sigma \in G_{K,n}$: 
\[
\zeta_{k_{\chi}K, n, S}(s,\sigma) := \sum_{(\fa, k_{\chi}K(\mu_{p^{n}})/k) = \sigma}N(\fa)^{-s}, 
\]
where $\fa$ runs over all integral ideals of $k$ coprime to all the primes in $S_{K}$ such that the Artin symbol $(\fa, k_{\chi}K(\mu_{p^{n}})/k)$ is equal to $\sigma$ and $N(\fa)$ denotes the norm of $\fa$. 
We put
\[
\theta_{k_{\chi}K,n,S} := \sum_{\sigma \in G_{K,n}}\zeta_{k_{\chi}K,n,S}(0,\sigma)\sigma^{-1}, 
\]
which is contained in $\bQ[G_{K,n}]$ (see \cite{Sie70}). 
The elements $\{\theta_{k_{\chi}K,n,S}\}_{n>0}$ are norm-coherent by \cite[Proposition~IV.1.8]{tatebook}. 
In addition, Deligne and Ribet proved in \cite{DR} that the element $e_{\omega\chi^{-1}}\theta_{k_{\chi}K,n,S}$ is contained in 
$\cO[G_{K,n}]e_{\omega\chi^{-1}}$. 
Here 
\[
e_{\omega\chi^{-1}} := \frac{1}{[k_{\chi}(\mu_{p}) \colon k]} \sum_{\sigma \in \Gal(k_{\chi}(\mu_{p})/k)}\omega\chi^{-1}(\sigma)\sigma^{-1}. 
\]
Hence we obtain an element 
\[
e_{\omega\chi^{-1}}\theta_{k_{\chi}K, \infty,S} := \varprojlim_{n>0}e_{\omega\chi^{-1}}\theta_{k_{\chi}K, n, S} \in \cO[[\Gal(k_{\chi}K(\mu_{p^{\infty}})/k)]]e_{\omega\chi^{-1}}. 
\]
Let 
\[
{\rm Tw} \colon \cO[[\Gal(k_{\chi}K(\mu_{p^{\infty}})/k)]] \longrightarrow  \cO[[\Gal(k_{\chi}K(\mu_{p^{\infty}})/k)]]
\]
denote the homomorphism induced by $\sigma \mapsto \chi_{\rm cyc}(\sigma)\sigma^{-1}$ for $\sigma \in \Gal(k_{\chi}K(\mu_{p^{\infty}})/k)$. 
Then we get an element
\[
\widetilde{L}_{p,K}^{\chi} :=  {\rm Tw}(e_{\omega\chi^{-1}}\theta_{k_{\chi}K, \infty, S}) \in \Lambda_{K}.  
\]
For each prime $\fq \not\in S$, we set 
\[
u_{\fq} := \chi(\widetilde{{\rm Fr}}_{\fq})^{-1}\chi_{\rm cyc}(\widetilde{{\rm Fr}}_{\fq})\widetilde{{\rm Fr}}_{\fq}^{-1} \in \Lambda_{K}^{\times}, 
\] 
where $\widetilde{{\rm Fr}}_{\fq} \in \Gal(k_{\infty}\cK/k)$ is the fixed lift of the arithmetic Frobenius at $\fq$.

\begin{definition}
For each field $K \in \Omega$, we define a modified $p$-adic $L$-function $L_{p,K}^{\chi} \in \Lambda_{K}$ by
\[
L_{p,K}^{\chi} := \left(\prod_{\fq \in S_{\rm ram}(K/k)}(-u_{\fq}) \right) \cdot \widetilde{L}_{p,K}^{\chi}. 
\]
\end{definition}

\begin{lemma}[{\cite[Lemma~3.5]{hres}}]\label{lem:rel2}
$(L_{p,K}^{\chi})_{K \in \Omega} \in \mathrm{ES}_{0}(T)$. 
\end{lemma}

\subsection{Iwasawa modules and characteristic ideals}

In this subsection, we introduce several Iwasawa modules and compute its characteristic ideals. 

We set 
\[
e_{\chi} := \frac{1}{[k_{\chi}(\mu_{p}) \colon k]} \sum_{\sigma \in \Gal(k_{\chi}(\mu_{p})/k)}\chi(\sigma)\sigma^{-1}. 
\]

\begin{definition}
Let  $K \in \Omega$ be a field. 
\begin{itemize}
\item[(i)] We write $M_{k_{\chi}K, \infty}$ for the maximal $p$-ramified pro-$p$ abelian extension of $k_{\chi}K(\mu_{p^{\infty}})$ and set 
\begin{align*}
X_{K}^{\chi} := e_{\chi}\left(\Gal(M_{k_{\chi}K,\infty}/k_{\chi}K(\mu_{p^{\infty}})) \otimes_{\bZ_{p}} \cO\right). 
\end{align*}
\item[(ii)] We write $M_{k_{\chi}K, S, \infty}$ for the maximal $S_{K}$-ramified pro-$p$ abelian extension of $k_{\chi}K(\mu_{p^{\infty}})$ and set 
\begin{align*}
X_{K, S}^{\chi} := e_{\chi}\left(\Gal(M_{k_{\chi}K,S,\infty}/k_{\chi}K(\mu_{p^{\infty}})) \otimes_{\bZ_{p}} \cO\right). 
\end{align*}
\item[(iii)] We write $N_{k_{\chi}K, \infty}$ for the maximal unramified pro-$p$ abelian extension of $k_{\chi}K(\mu_{p^{\infty}})$ and set 
\begin{align*}
Y_{K}^{\chi} := e_{\chi}\left(\Gal(N_{k_{\chi}K,\infty}/k_{\chi}K(\mu_{p^{\infty}})) \otimes_{\bZ_{p}} \cO\right). 
\end{align*}
\item[(iv)] We write $H_{K, p}^{\chi}$ for the $\Lambda_{K}$-submodule of $Y_{K}^{\chi}$ generated by the (finite) set of the Frobenius elements $\{{\rm Fr}_{\fp} \mid \fp \in S_{p}(k_{\chi}K(\mu_{\infty})) \}$. 
\end{itemize}
\end{definition}

Let $K \in \Omega$ be a field. 
By the weak Leopoldt conjecture proved by Iwasawa in \cite{Iwa73a}, the localization map 
\[
H^{1}(G_{k,S_{K}},\bT_{K}) \longrightarrow H^{1}(G_{k_{p}}, \bT_{K})
\]
is injective. Hence by global duality, we obtain the following two exact sequences of $\Lambda_{K}$-modules (see \cite[(4) and (5) in page~11]{hres}):
\begin{align}\label{exact:fundamental}
0 \longrightarrow H^{1}(G_{k,S_{K}},\bT_{K}) \longrightarrow H^{1}(G_{k_{p}}, \bT_{K}) \longrightarrow X_{K}^{\chi} 
\longrightarrow Y_{K}^{\chi}/H_{K,p}^{\chi} \longrightarrow 0, 
\end{align} 
\begin{align}\label{exact:diff}
0 \longrightarrow \bigoplus_{\fq \in S_{K} \setminus S} H^{1}(G_{k_{\fq}}, \bT_{K}) \longrightarrow 
X_{K,S}^{\chi} \longrightarrow X_{K}^{\chi} \longrightarrow 0. 
\end{align}
We put 
\[
C_{K} := {\rm coker}\left(H^{1}(G_{k,S_{K}},\bT_{K}) \longrightarrow H^{1}(G_{k_{p}}, \bT_{K}) \right). 
\]
Then the exact sequence~\eqref{exact:fundamental} splits into the following two short exact sequences 
\begin{align}\label{exact:1}
0 \longrightarrow H^{1}(G_{k,S_{K}},\bT_{K}) \longrightarrow H^{1}(G_{k_{p}}, \bT_{K}) \longrightarrow C_{K} \longrightarrow 0, 
\end{align} 
\begin{align}\label{exact:2}
0 \longrightarrow C_{K} \longrightarrow X_{K}^{\chi} 
\longrightarrow Y_{K}^{\chi}/H_{K,p}^{\chi} \longrightarrow 0. 
\end{align}

Since the $\Lambda_{K}$-module  $X_{K,S}^{\chi}$ does not have a non-trivial pseudo-null submodule for any field $K \in \Omega$, 
the following lemma is  proved by Iwasawa in \cite{Iwa73b}. 

\begin{lemma}\label{lemma:p-torsion}
If $X_{k}^{\chi} = X_{k,S}^{\chi}$ is a finitely generated $\bZ_{p}$-module, then  $X_{K,S}^{\chi}$ is a free $\bZ_{p}$-module of finite rank for any field $K \in \Omega$. 
\end{lemma}

\begin{proposition}[{\cite[Proposition~2.2]{Kur03}}]\label{prop:char-wiles}
Suppose that $X_{k}^{\chi}$ is a finitely generated $\bZ_{p}$-module. 
Then, for any field $K \in \Omega$, we have 
\[
 {\rm char}_{\Lambda_{K}}(X_{K,S}^{\chi}) = L_{p,K}^{\chi}  \Lambda_{K}. 
\]
\end{proposition}
\begin{proof}
Lemma~\ref{lemma:p-torsion} shows that $X_{K,S}^{\chi}$ is a free $\bZ_{p}$-module of finite rank.   
In this case, by using the Iwasawa main conjecture proved by Wiles in \cite{Wiles90}, Kurihara proved that 
\[
 {\rm Fitt}_{\Lambda_{K}}^{0}(X_{K,S}^{\chi}) = L_{p,K}^{\chi}  \Lambda_{K}
\]
(see \cite[Proposition~2.2]{Kur03}). 
Since $L_{p,K}^{\chi}$ is a regular element of $\Lambda_{K}$, we have 
\[
{\rm Fitt}_{\Lambda_{K}}^{0}(X_{K,S}^{\chi}) =  {\rm Fitt}_{\Lambda_{K}}^{0}(X_{K,S}^{\chi})^{**}. 
\] 
Hence by Lemmas~\ref{lemma:ref}~and~\ref{lemma:bidualeq}, it suffices to show that 
\[
 {\rm Fitt}_{\Lambda_{K}}^{0}(X_{K,S}^{\chi}) \Lambda_{K, \fr} =  {\rm char}_{\Lambda_{K}}(X_{K,S}^{\chi}) \Lambda_{K,\fr}
\]
for any prime $\fr$ of $\Lambda_{K}$ with ${\rm ht}(\fr) \leq 1$. If $p \not\in \fr$, then the ring $\Lambda_{K, \fr}$ is regular, and hence  Corollary~\ref{cor:char-fitt} shows that 
\begin{align*}
{\rm Fitt}_{\Lambda_{K}}^{0}(X_{K,S}^{\chi}) \Lambda_{K, \fr} &=  L_{p,K}^{\chi}\Lambda_{K, \fr} 
\\
&= (L_{p,K}^{\chi}\Lambda_{K, \fr})^{**} 
\\
&= {\rm char}_{\Lambda_{K,\fr}}(X_{K,S}^{\chi} \otimes_{\Lambda_{K}}\Lambda_{K,\fr})  
\\
&= {\rm char}_{\Lambda_{K}}(X_{K,S}^{\chi}) \Lambda_{K,\fr}. 
\end{align*}
Suppose that  $p \in \fr$. 
In this case, the module $X_{K,S}^{\chi} \otimes_{\Lambda_{K}}\Lambda_{K,\fr}$ vanishes by Lemma~\ref{lemma:torsion} since $p$ is a $X_{K,S}^{\chi}$-regular element. 
Hence we have  
\[
{\rm Fitt}_{\Lambda_{K}}^{0}(X_{K,S}^{\chi}) \Lambda_{K, \fr} =  \Lambda_{K, \fr} = {\rm char}_{\Lambda_{K}}(X_{K,S}^{\chi}) \Lambda_{K,\fr}. 
\] 
\end{proof}

We note that $f_{K}^{-1}L_{p, K}^{\chi} \in \Lambda_{K}[1/p]$ by Proposition~\ref{prop:key_div} and Lemma~\ref{lem:rel2}. 

\begin{proposition}\label{prop:char}
For any field $K \in \Omega$, we have 
\[
{\rm char}_{\Lambda_{K}}(X_{K}^{\chi})\Lambda_{K}[1/p] = f_{K}^{-1}L_{p, K}^{\chi}  \Lambda_{K}[1/p]. 
\]
\end{proposition}
\begin{proof}
Since $\Lambda_{K}[1/p]$ is the finite product of principal ideal domain, 
the notion of the characteristic ideal  coincides with the usual one by Corollary~\ref{cor:char-fitt}. 
In particular, the characteristic ideal is  additive in short exact sequences of finitely generated $\Lambda_{K}[1/p]$-modules. 
Hence the exact sequence~\eqref{exact:diff} shows that 
\[
{\rm char}_{\Lambda_{K}}(X_{K,S}^{\chi})\Lambda_{K}[1/p] 
= {\rm char}_{\Lambda_{K}}(X_{K}^{\chi}) \prod_{\fq \in S_{\rm ram}(K/k) } {\rm char}_{\Lambda_{K}}(H^{1}(G_{k_{\fq}}, \bT_{K}))\Lambda_{K}[1/p]. 
\]
Thus this proposition follows from Propositions~\ref{prop:char-local}~and~\ref{prop:char-wiles}. 
\end{proof}

We recall the conjecture concerning the structure of $Y_{K}^{\chi}$ proposed by Greenberg in \cite{Gre}.

\begin{conjecture}[Greenberg]\label{conj:greenberg}
The $\Lambda_{K}$-module $Y_{K}^{\chi}$ is pseudo-null for any field $K \in \Omega$.  
\end{conjecture}

\subsection{On the structure of the module of Euler systems for $\bG_{m}$}

Recall that $r = [k \colon \bQ]$. 
Suppose that 
\[
H^{0}(G_{k_{p}},T/\fm T) = H^{2}(G_{k_{p}},T/\fm T) = 0. 
\]
Then, for any field $K \in \Omega$, the complex ${\bf R}\Gamma(G_{k_{p}}, \bT_{K})$ has perfect amplitude in $[1, 1]$.  
Since the Euler--Poincare characteristic of ${\bf R}\Gamma(G_{k_{p}}, \bT_{K})$ is $r$, we have an isomorphism 
\[
H^{1}(G_{k_{p}}, \bT_{K}) \cong \Lambda_{K}^{r}
\] 
such that the following diagram commutes for any field $L \in \Omega$ with $L \subset K$: 
\[
\xymatrix{
H^{1}(G_{k_{p}}, \bT_{K}) \ar[r]^-{\cong} \ar[d] & \Lambda_{K}^{r} \ar[d]
\\
H^{1}(G_{k_{p}}, \bT_{L}) \ar[r]^-{\cong}  & \Lambda_{L}^{r}, 
}
\]
where the left vertical arrow is induced by the canonical homomorphism $\bT_{K} \longrightarrow \bT_{L}$ and the right vertical arrow is  the canonical projection.  
The localization map at $p$ induces a homomorphism 
\[
{\bigcap}^{r}_{\Lambda_{K}}H^{1}(G_{k,S_{K}}, \bT_{K}) \longrightarrow {\bigcap}^{r}_{\Lambda_{K}}H^{1}(G_{k_{p}}, \bT_{K}) \cong {\bigcap}^{r}_{\Lambda_{K}}\Lambda_{K}^{r} = \Lambda_{K}, 
\]
and we obtain a homomorphism 
\[
{\rm ES}_{r}(T) = {\rm ES}\left(\left\{{\bigcap}^{r}_{\Lambda_{K}}H^{1}(G_{k,S_{K}}, \bT_{K})\right\}_{K \in \Omega} \right) \longrightarrow  {\rm ES}(\{\Lambda_{K}\}_{K \in \Omega}\}) = {\rm ES}_{0}(T). 
\]

\begin{proposition}[{\cite[Proposition~2.10]{hres}}]\label{prop:image}
The homomorphism ${\rm ES}_{r}(T) \longrightarrow {\rm ES}_{0}(T)$ is injective and we have 
\[
{\rm im}\left({\rm ES}_{r}(T) \longrightarrow {\rm ES}_{0}(T)\right) = {\rm ES}_{0}(T) \cap \prod_{K \in \Omega}{\rm char}_{\Lambda_{K}}(C_{K}). 
\]
\end{proposition}
\begin{proof}
Let $K$ be a field. 
Recall that, by \eqref{exact:1}, we have an exact sequence of $\Lambda_{K}$-modules
\[
0 \longrightarrow H^{1}(G_{k,S_{K}},\bT_{K}) \longrightarrow \Lambda_{K}^{r} \longrightarrow C_{K} \longrightarrow 0. 
\] 
Hence Lemma~\ref{lemma:bidual-inj} shows that the homomorphism ${\rm ES}_{r}(T) \longrightarrow {\rm ES}_{0}(T)$ is injective. 
Furthermore, the definition of the characteristic ideal shows that 
\[
{\rm im}\left({\bigcap}^{r}_{\Lambda_{K}}H^{1}(G_{k,S_{K}}, \bT_{K}) \longrightarrow  {\bigcap}^{r}_{\Lambda_{K}}\Lambda_{K}^{r} = \Lambda_{K}\right) = {\rm char}_{\Lambda_{K}}(C_{K}), 
\]
which completes the proof. 
\end{proof}


\begin{theorem}\label{thm:main}
Suppose that 
\begin{itemize}
\item both $H^{0}(G_{k_{p}},T/\fm T)$ and $H^{2}(G_{k_{p}},T/\fm T)$ vanish, 
\item the module $X_{k}^{\chi}$ is a finitely generated $\bZ_{p}$-module, 
\item Greenberg conjecture (Conjecture~\ref{conj:greenberg}) holds true. 
\end{itemize}
Then the $\cO[[\Gal(k_{\infty}\cK/k)]]$-module ${\rm ES}_{r}(T)$ is free of rank $1$. 
Furthermore, there exists a basis $\{c_{K}\}_{K \in \Omega} \in {\rm ES}_{r}(T)$ such that 
its image under the injection ${\rm ES}_{r}(T) \longhookrightarrow {\rm ES}_{0}(T)$ is $\{L_{p,K}^{\chi}\}_{K \in \Omega}$. 
\end{theorem}
\begin{proof}
Since we assume  Greenberg conjecture, 
by the exact sequence~\eqref{exact:2} and Corollary~\ref{cor:pseudo}, 
we have 
\[
{\rm char}_{\Lambda_{K}}(X_{K}^{\chi}) = {\rm char}_{\Lambda_{K}}(C_{K}) 
\] 
for any field $K \in \Omega$. 
Hence, by Proposition~\ref{prop:image}, it suffices to show that the homomorphism 
\begin{align}\label{hom2}
\cO[[\Gal(k_{\infty}\cK/k)]] \longrightarrow 
{\rm ES}_{0}(T) \cap \prod_{K \in \Omega}{\rm char}_{\Lambda_{K}}(X_{K}^{\chi}); \, \lambda \mapsto \{\lambda L_{p,K}^{\chi}\}_{K \in \Omega}
\end{align}
is an isomorphism. 
Since $L_{p,K}^{\chi}$ is a regular element of $\Lambda_{K}$ for any field $K \in \Omega$, the homomorphism~\eqref{hom2} is injective. 

To show the surjectivity of the homomorphism~\eqref{hom2}, take an Euler system 
\[
(c_{K})_{K \in \Omega} \in {\rm ES}_{0}(T) \cap \prod_{K \in \Omega}{\rm char}_{\Lambda_{K}}(X_{K}^{\chi}).
\] 
Let $K \in \Omega$ be a field. 
Proposition~\ref{prop:char} shows that there is an element $\alpha_{K} \in \Lambda_{K}[1/p]$ such that 
\[
c_{K} = \alpha_{K}f_{K}^{-1}L_{p,K}^{\chi}. 
\]
Lemmas~\ref{lemma:euler-rel}~and~\ref{lem:rel2} imply that $\{f_{K}^{-1}L_{p,K}^{\chi}\}_{K \in \Omega}$ are norm-coherent. 
Hence for any field $L \in \Omega$ with $L \subset K$, we have 
\begin{align*}
\pi_{K,L}(\alpha_{K})f_{L}^{-1}L_{p,L}^{\chi} 
&= \pi_{K,L}(c_{K}) 
\\
&= \left(\prod_{\fq \in S_{\rm ram}(K/k) \setminus S_{\rm ram}(L/k) }P_{\fq}({\rm Fr}_{\fq}^{-1})\right)c_{L} 
\\
&= \left(\prod_{\fq \in S_{\rm ram}(K/k) \setminus S_{\rm ram}(L/k)}P_{\fq}({\rm Fr}_{\fq}^{-1})\right)\alpha_{L}f_{L}^{-1}L_{p,L}^{\chi}. 
\end{align*}
Since $f_{K}^{-1}L_{p,K}^{\chi}$ is a regular element of $\Lambda_{K}[1/p]$, we conclude that $(\alpha_{K})_{K \in \Omega} \in {\rm ES}(\{\Lambda_{K}[1/p]\}_{K \in \Omega})$. 
Therefore, Proposition~\ref{prop:key_div} shows that 
\[
\alpha_{K} \in f_{K} \Lambda_{K}[1/p], 
\]
and we have 
\[
c_{K} \in \Lambda_{K} \cap L_{p,K}^{\chi}  \Lambda_{K}[1/p].  
\]
We note that $ {\rm char}_{{\Lambda}_{K}}(X_{K,S}^{\chi}) = L_{p,K}^{\chi}  \Lambda_{K}$ by Proposition~\ref{prop:char-wiles}.  
The module $X_{K,S}^{\chi}$ is $p$-torsion-free  by Lemma~\ref{lemma:p-torsion} since we assume that $X_{k}^{\chi}$ is a finitely generated $\bZ_{p}$-module. 
Hence Lemma~\ref{lemma:torsion} implies that 
\[
\Lambda_{K} \cap L_{p,K}^{\chi}  \Lambda_{K}[1/p] = \bigcup_{m>0}p^{-m}\left(p^{m}\Lambda_{K} \cap L_{p,K}^{\chi}  \Lambda_{K} \right) = L_{p,K}^{\chi}  \Lambda_{K}. 
\]
This shows that $c_{K}  \in L_{p, K}  \Lambda_{K}$, and hence the homomorphism~\eqref{hom2} is surjective.  
 \end{proof}

\subsection{Vertical determinantal systems }

For each integer $n \geq 0$, let $k_{n}$ denote the field $k_{n} \subset k_{\infty}$ satisfying $[k_{n} \colon k] = p^{n}$. 
For any field $K\in \Omega$, put 
\[
K_{n} := k_{n}K \, \textrm{ and  } \, T_{K_{n}} := \mathrm{Ind}_{G_{k}}^{G_{k_{n}K}}(T). 
\]

\begin{definition}[{\cite[Definition~2.9]{sbA}}]
For any field $K\in \Omega$ and any integer $n \geq 0$, we set 
\begin{align*}
{\bf R}\Gamma_{c}(G_{k, S_{K_{n}}}, T_{K_{n}}) := {\bf R}\Hom_{\cO}({\bf R}&\Gamma(G_{k, S_{K_{n}}}, T_{K_{n}}), \cO)[-3] 
\\
&\oplus \left(\bigoplus_{v \in S_{\infty}(k)} H^{0}(G_{k_{v}}, T_{K_{n}})\right)[-1]. 
\end{align*}
\end{definition}

\begin{definition}
Let $K\in \Omega$ be a field and let $n \geq 0$ be an integer. 
For any $\bZ_{p}[\Gal(k_{n}K/k)]$-module $M$, we define a new $\bZ_{p}[\Gal(k_{n}K/k)]$-module $M^{\iota}$ by 
\[
M^{\iota} := M \otimes_{\bZ_{p}[\Gal(k_{n}K/k)], \iota} \bZ_{p}[\Gal(k_{n}K/k)]. 
\]
Here $\iota \colon \bZ_{p}[\Gal(k_{n}K/k)] \longrightarrow \bZ_{p}[\Gal(k_{n}K/k)]$ is the homomorphism defined by $g \mapsto g^{-1}$ for $g \in \Gal(k_{n}K/k)$. 
\end{definition}

As explained in \cite[Definition~2.9]{sbA}, we have the canonical homomorphism 
\[
{\det}^{-1}({\bf R}\Gamma_{c}(G_{k, S_{K_{m}}}, T_{K_{m}}))^{\iota} \longrightarrow {\det}^{-1}({\bf R}\Gamma_{c}(G_{k, S_{L_{n}}}, T_{L_{n}}))^{\iota}
\]
for any field $L \in \Omega$ with $L \subset K$ and integer $m$ with $n \leq m$. 
We then define the $\cO[[\Gal(k_{\infty}\cK/k)]]$-module of vertical determinantal systems by 
\[
{\rm VS}(T) := \varprojlim_{K \in \Omega, \, n \geq 0} {\det}^{-1}({\bf R}\Gamma_{c}(G_{k, S_{K_{n}}}, T_{K_{n}}))^{\iota}. 
\]

\begin{proposition}[{\cite[Proposition~2.10]{sbA}}]\label{prop:VS}
The $\cO[[\Gal(k_{\infty}\cK/k)]]$-module  ${\rm VS}(T)$ is free of rank $1$. 
\end{proposition}

\begin{remark}
As Burns and Sano mentioned in \cite[Remark~2.11]{sbA}, the equivariant Tamagawa number conjecture predicts the existence of  a unique basis of ${\rm VS}(T)$   such that  its image  under a period-regulator isomorphism is 
the leading term of  an $L$-series. 
\end{remark}

In \cite[Theorem~2.18]{sbA}, Burns and Sano showed that there is a natural homomorphism 
\[
{\rm VS}(T) \longrightarrow {\rm ES}_{r}(T). 
\]
According to the equivariant Tamagawa number conjecture, Euler systems in the image of this homomorphism should be related to  $L$-values. 
Hence it is important  to understand the size of the image of this homomorphism.

\begin{theorem}\label{thm:main2}
Suppose that 
\begin{itemize}
\item both $H^{0}(G_{k_{p}},T/\fm T)$ and $H^{2}(G_{k_{p}},T/\fm T)$ vanish, 
\item the module $X_{k}^{\chi}$ is a finitely generated $\bZ_{p}$-module, 
\item  Greenberg conjecture (Conjecture~\ref{conj:greenberg}) holds true. 
\end{itemize}
Then the canonical homomorphism ${\rm VS}(T) \longrightarrow {\rm ES}_{r}(T)$ is an isomorphism. 
\end{theorem}

Before we give a proof of Theorem \ref{thm:main2}, we introduce the modules of Kolyvagin and Stark systems.  

Recall that $\overline{T}$ denotes the residual representation of $T$ and $r = [k\colon \bQ]$.  
Let ${\rm KS}_{r}(\overline{T})$ and ${\rm SS}_{r}(\overline{T})$ denote the modules of Kolyvagin and Stark systems of rank $r$ associated with the canonical Selmer structure on $\overline{T}$, respectively (see \cite[Definition~3.2.1]{MRkoly} for the definition of the canonical Selmer structure and \cite[Definitions~3.1~and~4.1]{sbA} for the definition of Kolyvagin and Starks systems). 
In the paper \cite{sbA}, Burns and Sano proved the following 

\begin{proposition}[{\cite[Theorems~3.12~and~4.16]{sbA}}]\label{prop:sbA}
We have the following commutative diagram:
\begin{align*}
\xymatrix{
{\rm VS}(T) \ar[r] \ar[d] & {\rm ES}_{r}(T) \ar[r]  & {\rm KC}_{r}(\overline{T})
\\
{\rm SS}_{r}(\overline{T}) \ar[rr]  & & {\rm KS}_{r}(\overline{T}). \ar@{^{(}->}[u]
}
\end{align*}
Here ${\rm KC}_{r}(\overline{T})$ denotes the module of Kolyvagin collections of rank $r$ (see \cite[Definition~4.11]{sbA}). 
Furthermore, if the module $H^{2}(G_{k_{p}}, \overline{T})$ vanishes, then the homomorphism ${\rm VS}(T) \longrightarrow {\rm SS}_{r}(\overline{T})$ is surjective. 
\end{proposition}

Furthermore, Barns, the author, and Sano proved in the paper \cite{bss} the following 

\begin{proposition}[{\cite[Theorems~5.2~(ii)~and~6.12]{bss}}]\label{prop:bss}\ 
\begin{itemize}
\item[(i)] The image of the homomorphism ${\rm ES}_{r}(T) \longrightarrow {\rm KC}_{r}(\overline{T})$ is contained in ${\rm KS}_{r}(\overline{T})$. 
\item[(ii)]  If the module $H^{2}(G_{k_{p}}, \overline{T})$ vanishes, the homomorphism ${\rm SS}_{r}(\overline{T}) \longrightarrow  {\rm KS}_{r}(\overline{T})$ is an isomorphism and ${\rm KS}_{r}(\overline{T})$ is a free $\cO/\fm$-module of rank $1$. 
\end{itemize}
\end{proposition}
\begin{proof}
Claim~(i) follows from \cite[Theorem~6.12]{bss}. 
Since claim~(ii) follows from \cite[Theorem~5.2]{bss}, we only need to check that \cite[Hypotheses~3.2,~3.3,~and~4.2]{bss} are satisfied. 
\cite[Lemma~6.1.5]{MRkoly} implies that $\overline{T}$ satisfies \cite[Hypotheses~3.2~and~3.3]{bss}. 
By \cite[Lemma~3.7.1~(i)]{MRkoly}, the canonical Selmer structure $\cF_{\rm can}$ on $\overline{T}$ is cartesian.
Hence \cite[Lemma~6.6]{MRselmer} shows that if the core rank $\chi(T, \cF_{\rm can})$ is $[k \colon \bQ]$, then \cite[Hypotheses~4.2]{bss} is satisfied. 
By \cite[Theorem~5.2.15]{MRkoly}, we have 
\[
\chi(T, \cF_{\rm can}) = \sum_{v \in S_{\infty}(k)}{\rm rank}_{\cO} \left( H^{0}(G_{k_{v}}, T^{*}(1)) \right) +  {\rm rank}_{\cO}\left( H^{2}(G_{k_{p}}, T)^{\vee} \right). 
\]
Here $T^{*} := \Hom_{\cO}(T, \cO)$. 
Since we assume that the module $H^{2}(G_{k_{p}}, \overline{T})$ vanishes, we have $H^{2}(G_{k_{p}}, T) = 0$. Furthermore, the fact that $\chi$ is an even character implies that $H^{0}(G_{k_{v}}, T^{*}(1))  \cong \cO$. 
Hence since $k$ is a totally real field, we conclude that $\chi(T, \cF_{\rm can}) = \# S_{\infty}(k) = [k \colon \bQ]$. 
\end{proof}

\begin{proof}[proof of Theorem \ref{thm:main2}]
To simplify the notation, we put $\cR := \cO[[\Gal(k_{\infty}\cK/k)]]$, and $\fm_{\cK}$ denotes the maximal ideal of $\cR$. 
Propositions~\ref{prop:sbA}~and~\ref{prop:bss} show that there is the following commutative diagram 
\begin{align*}
\begin{split}
\xymatrix{
{\rm VS}(T) \ar[r]^-{\varphi} \ar@{->>}[rd] & {\rm ES}_{r}(T) \ar[d]
\\
  &  {\rm KS}_{r}(\overline{T}), 
}
\end{split}
\end{align*}
where ${\rm VS}(T) \longrightarrow {\rm KS}_{r}(\overline{T})$ is surjective. 
Since the $\cR$-modules ${\rm VS}(T)$ and ${\rm ES}_{r}(T)$ are free of rank $1$ by Theorem~\ref{thm:main} and Proposition \ref{prop:VS}, 
one can take elements $v \in {\rm VS}(T)$ and $c \in {\rm ES}_{r}(T)$ satisfying 
\[
\cR v = {\rm VS}(T) \,\,\, \textrm{ and } \,\,\,  \cR c = {\rm ES}_{r}(T). 
\]
Then there exists an element $f \in \cR$ such that 
\[
\varphi(v) = fc.  
\]
Since the homomorphism $\cR v = {\rm VS}(T) \longrightarrow {\rm KS}_{r}(\overline{T})$ is surjective and ${\rm KS}_{r}(\overline{T}) \cong \cR/\fm_{\cR}$ by Proposition \ref{prop:bss}, we see that $f \not\in \fm_{\cR}$, and hence $f$ is unit.  This shows that $\varphi$ is an isomorphism. 
\end{proof}

\end{document}